\documentclass[11pt,reqno]{amsart}
\usepackage{amssymb,amsmath}\newcommand{\be}{\begin{eqnarray}}
\newcommand{\ee}{\end{eqnarray}}\newcommand{\card}{\#}\newcommand{\eps}{{\mbox{$\epsilon$}}}\newcommand{\e}{{\varepsilon}}\newcommand{\R}{{\mathbb R}}\newcommand{\D}{{\mathcal D}}\newcommand{\K}{{\mathcal K}}

\newcommand{\Fav}{\operatorname{Fav}}\newtheorem{theorem}{Theorem}\newtheorem{lemma}[theorem]{Lemma}\newtheorem{cor}[theorem]{Corollary}\newtheorem{prop}[theorem]{Proposition}\theoremstyle{definition}\theoremstyle{remark}\numberwithin{equation}{section}\input epsf.sty
\begin{document}\thispagestyle{empty}

%
%
%
%
%
%
%
%
%
%
%
%
\newcommand{\G}{{\mathcal G}}
\newcommand{\C}{{\mathbb C}}
\newcommand{\Pa}{{P_{1,\theta}(y)}}
\newcommand{\Pb}{{P_{2,\theta}(y)}}
\newcommand{\Pshp}{{P_{1,t}^\sharp(x)}}
\newcommand{\Pflt}{{P_{1,t}^\flat(x)}}
\newcommand{\OTL}{{\Omega(\theta,\ell)}}
\newcommand{\rsz}{{R(\theta^*)}}
\newcommand{\phitil}{{\tilde{\varphi}_t}}
\newcommand{\lambdatil}{{\tilde{\lambda}}}

\title[Buffon's needle power law for Sierpinski gasket]{{The power law for Buffon's needle landing near the Sierpinski gasket}}
\author{Matthew Bond}\address{Matthew Bond, Dept. of Math., Michigan State University.
{\tt bondmatt@msu.edu}}
\author{Alexander Volberg}\address{Alexander Volberg, Dept. of  Math., Michigan State Univ. 
and the University of Edinburgh.
{\tt volberg@math.msu.edu}}

\thanks{Research of the authors was supported in part by NSF grants  DMS-0501067, 0758552 }
\begin{abstract}In this paper we   get a power estimate from above of the probability that Buffon's needle will land within distance $3^{-n}$ of Sierpinski's gasket of Hausdorff dimension 1. In comparison with the case of $1/4$ corner Cantor set considered in  Nazarov, Peres, and the second author \cite{NPV}: we still need the technique of \cite{NPV} for splitting the directions to good and bad ones, but the case of Sierpinski gasket is considerably more generic and lacks symmetry, resulting in a need for much more careful estimates of zeros of the Fourier transform of Cantor measure.
\end{abstract}
\maketitle

\section{{\bf Introduction}} \label{sec:intro}
Among self-similar planar sets of Hausdorff dimension 1, some of the simplest are the Sierpinski gasket $\G$ (formed by three self-similarities by the scaling factor 1/3) and the square 4-corner Cantor set $\K$ (formed by four self-similarities by the scaling factor 1/4; it is a cartesian product of two Cantor sets in $\R$). By the Besicovitch projection theorem \cite{mattila3}, these irregular sets of finite Hausdorff $H^1$ measure must have zero length in almost every orthogonal projection onto a line. One may partially construct these sets in the usual way by taking their convex hulls and then taking the union of all possible images of $n$-fold compositions of the similarity maps, which we call $\G_n$ and $\K_n$, respectively. Then $\G=\bigcap_n\G_n$ and $\K=\bigcap_n\K_n$. One may then ask the rate at which the Favard length -- the average over all directions of the length of the orthogonal projection onto a line in that direction -- of these sets $\G_n$ and $\K_n$ decay to zero as a function of $n$(\footnote{Such decay must occur by the Besicovitch projection theorem and by continuity of measures, since we're taking the Lebesgue measure of decreasing sets in the parameter space of $\{\text{directions}\}\times\{\text{projected x values}\}$.}). For bounded sets, Favard length is also called Buffon needle probability, since up to a normalization constant, it is the likelihood that a long needle dropped with independent, uniformly distributed orientation and distance from the origin will intersect the set somewhere. Observe that $\G_n$ and $\K_n$ are in some sense comparable to small neighborhoods of $\G$ and $\K$, so that $\Fav(\G_n)$ is comparable to the likelihood that ``Buffon's needle" will land in a $3^{-n}$-neighborhood of $\G$.

The first quantitative results for the Favard length problem were obtained in \cite{PS},\cite{T}; in the latter paper a general way of making a quantitative statement from the Besicovitch theorem is considered. But being rather general, this method does not give a good estimate for self-similar structures such as $\K_n$ or $\G_n$.

Indeed, vastly improved estimates have been proven in these cases: in \cite{NPV}, it was shown that for $p<1/6$, $Fav(\K_n)\leq\frac{c_p}{n^{p}}$, and the current paper extends this result to $\G_n$ for some other $p>0$. These results cannot possibly be improved to $p=1$: $Fav(\K_n)\geq c\frac{\log\,n}{n}$ (This is \cite{BV}(\footnote{the method is stable under ``bending the needle" slightly - see \cite{BV2}.}), and the argument and result also apply to $\G_n$.) Compare this with \cite{PS}, in which it was shown that certain random sets of which $\K_n$ is a special case almost surely decay in Favard length like $\frac1{n}$.

Crucial to \cite{BV} was a tiling property: namely, under orthogonal projection on the line with slope $1/2$, the squares composing $\K_n$ tile a line segment. Oddly enough, such a property can be used to prove upper bounds as well: under the assumption that some orthogonal projection in some direction contains an interval, Laba and Zhai \cite{LZ} showed that the result of \cite{NPV} holds for Cantor-like product sets of finite $H^1$ measure (but with a smaller exponent). Their argument uses tiling results obtained in Kenyon \cite{kenyon} and Lagarias-Wang \cite{lawang} to fill in a gap where \cite{NPV} fails to generalize (more on this shortly).

With the exception of \cite{PS} and \cite{T}, the above papers all extract their results from information about $L^2$ norms of the projection multiplicity function, which counts how many squares (or triangles) project to cover each point. The function $f_{n,\theta}:\R\to\mathbb N$ is defined by $$f_{n,\theta}=\sum_{\text{Sierpinski triangles T of }\G_n}{\chi_{proj_\theta(T)}}.$$ Note that $Fav(\G_n)=\pi^{-1}\int_0^{\pi}|supp(f_{n,\theta})|d\theta$. In \cite{NPV} and \cite{BV}, the $L^2$ norm of the analog of this function for squares was studied to obtain Buffon needle probability estimates for $\K_n$ -- in \cite{BV}, $p=1,2$ were related to $\chi_{supp(f_{n,\theta})}$ via the Cauchy inequality, while in \cite{NPV}, $p=2$ was studied via Fourier transforms and related to the measure of the level sets of $f$.

Consider some heuristics. Let $f:[0,1]\to\mathbb{N}$ be any sum of measurable characteristic fuctions such that $||f||_{L^1}=1$. If the mass is concentrated on a small set, the $L^p$ norm should be large for $p>1$. Thus a large $L^p$ norm should indicate that the support of a function is small, and vice versa. Let $K>0$, let $A=supp\{f\}$, and let $A_K =\{x:f\geq K\}$. $1=\int f\leq ||f||_p||\chi_A||_q$, so $m(A)\geq ||f||_p^{-q}$, a decent estimate. The other basic estimate is not so sharp: $m(A)\leq 1-(K-1)m(A_K)$. However, a combinatorial self-similarity argument of \cite{NPV} shows that for the Favard length problem, it bootstraps well under further iterations of the similarity maps - this argument is revisited in Section $\ref{combinatorics}$. Hence, up to some loss of sharpness, it has been shown that to study Favard length of these self-similar sets, it is necessary and sufficient to study the $L^2$ norms of $f_{n,\theta}$.(\footnote{So far, only $L^p$ for $p=1,2,$ or $\infty$ have played any useful role, to our knowledge.})

One must average $|supp{f_{n,\theta}}|$ over the parameter $\theta$ to get Favard length of $\G_n$, and for some directions, the orthogonal projections do not even decay to length zero with $n$ (i.e., the $L^2$ norms of $f_{n,\theta}$ are bounded for these angles), and this countable dense set of directions is to a large extent classified in \cite{kenyon}. In \cite{NPV}, a method for controlling the measure of a set of angles $E$ on which the projections fail to decay rapidly was found: one takes the Fourier transform of $f_{n,\theta}$ in the length variable, and takes a sample integral of $|\hat{f}_{n,\theta}(x)|^2$ over a chosen small interval $I$ where $\int_{E\times I}|\hat{f}_{n,\theta}(x)|^2d\theta dx$ is small. One then shows that there is a $\theta\in E$ such that $\int_I|\hat{f}_{n,\theta}(x)|^2dx$ is not too small relative to the size of $E$, and so $E$ must be small.

In all cases, $\hat{f}_{n,\theta}$ is a self-similar product $\prod_k\varphi_\theta(3^{-k}y)$ of trigonometric polynomials $\varphi_\theta$. The danger is that the low-frequency zeroes might kill off the better-behaved high-frequency terms. In \cite{NPV}, the four frequencies of $\varphi_\theta$ were symmetric around 0, allowing the terms to simplify to two cosines, and trig identities allowed the whole product to be estimated by a single sine term. In \cite{LZ}, an analogous role was played by tiling, and the product structure allowed for a change and separation of variables. In the current case, $\G_n$, neither of these things happen, but our considerations show that a so-called ``analtyic tiling" on the Fourier side (Section $\ref{sec:Analytic tiling}$) proves that the complex zeroes from different factors are separated away from each other, preventing any resonance that may have caused the set of small values to grow too large.

Separating variables is more difficult when there is no product structure, so instead we isolated the zeroes in small intervals and found estimates valid for each small interval around each zero, so that an estimate on medium-frequency terms could be made independent of $x$. These zeroes $\lambda_j$ depend on $\theta$, so we traced how the zeroes $\lambda_j$ of $\varphi$ move as $\theta$ varies. In order for our estimates to work, we needed the real parts of the $\lambda_j$ to move at a more or less constant rate without too many oscillations, allowing the path integral of the Riesz product to be controlled by the basic integral of the Riesz product on $[0,2\pi]$. These technicalities are resolved in Section \ref{complex}. To get such highly regular behavior in the zeroes $\lambda_j$ as functions of $\theta$, we had to consider them as functions of a complex variable $\zeta=\theta+i\sigma$ and appeal to holomorphic function theory.

The case of the gasket is much closer to the generic self-similar case as $\varphi$ becomes a rather general $3$-term exponential sum, providing a much better glimpse at the general Besicovitch irregular self-similar set than the sets considered in \cite{NPV} and \cite{LZ}. We believe that using this approach one can work with all such sets. However, there are a couple of problems which remain unresolved for now: see Section $\ref{discu}$.

Rather strangely, a claim in the spirit of the Carleson Embedding Theorem, in the form of Lemma \ref{CETSQ}, plays an important part in our reasoning. Because the Fourier transform turns stacks of triangles (i.e., sums of overlapping characteristic functions) into clusters of frequencies, this lemma provides important upper bounds when $\theta$ belongs to $E$.

The main result of this article is the following estimate.
\begin{theorem}
\label{mainth1}For each $p<<1$, there exists $C_p>0$ such that for all $n\in\mathbb{N}$,
$$
Fav(\G_n)\le\frac{C_p}{n^p}\,.
$$
\end{theorem}

The exponent can be made explicit, but it is somewhat technical to track everything. See Section $\ref{discu}$, in which degenerate gaskets are also considered. The techniques of this paper can also be used to prove a weaker result in a more general setting. The reason for the weaker result is also discussed in Section $\ref{discu}$, but it is not known whether the strong result is in fact false in this setting.

\begin{theorem}\label{mainth2}
Let $T_j:\C\to\C$, $j=1,...,M$, be self-similarity mappings $T_j(z)=\frac1{M}z+c_j$ with $c_j\in\C$ not colinear. Suppose also that the $T_j$ satisfy the open set condition with the open set $U$ (as in \cite{mattila3}). Let $E_n$ be the union of all possible images of $U$ under $n$-fold compositions of self-similarity maps chosen from $\{T_j\}$. Then there are constants $c$, $C$ such that
$$Fav(E_n)\leq Ce^{-c\sqrt{\log\,n}}.$$
\end{theorem}

We omit the proof of Theorem $\ref{mainth2}$. The proof of Theorem $\ref{mainth1}$ mutatis mutandis, except there are some extra difficulties which appear in Section $\ref{complex}$ and some estimates are weakened in the absense of an easy analog of Section $\ref{sec:Analytic tiling}$. There seems to be a good chance that a power estimate is again true in this general setting, but whether this is the case remains to be seen.

\medskip


\section{{The Fourier-analytic part}} \label{sec:fourier}
\subsection{{The setup}} \label{subsec:setup}
The goal of this section is to prove Theorem $\ref{mofE}$, which shows that for most directions, a considerable amount of stacking occurs when the triangles are projected down. Throughout the paper, the constants $c$ and $C$ will vary from line to line, but will be absolute constants not depending on anything. The symbols $c$ and $C$ will typically denote constants that are sufficiently small or large, respectively. Everywhere we use the definition $B(z_0,\varepsilon):=\{z\in\C:|z-z_0|<\varepsilon\}$.

For convenience, we will now rescale $\G_n$ by a factor absolutely comparable to 1 and bound the triangles by discs and study this set instead. That is, for $\alpha\in\lbrace -1, 0, 1\rbrace^{n+1}$ let $$z_\alpha:=\sum_{k=1}^n{(\frac{1}{3})^ke^{i\pi[\frac{1}{2}+\frac{2}{3}\alpha_k]}},$$ and then let $$\G_n:=\bigcup_{\alpha\in\lbrace -1, 0, 1\rbrace^n}B(z_\alpha,3^{-n}).$$ Note that $\G_n$ has $3^n$ discs of radius $3^{-n}.$ After a rescaling, the usual $n+1$st Sierpinski gasket (composed of $3^{n+1}$ triangles) sits inside of $\G_n$. We may still speak of the approximating discs as ``Sierpinski triangles."

Observe that $f_{n,\theta}=\nu_n *3^n\chi_{[-3^{-n},3^{-n}]},$ where $\nu_n :=*_{k=1}^n\widetilde{\nu}_k$ and $$\widetilde{\nu}_k=\frac{1}{3}[\delta_{3^{-k}\cos(\pi/2 -\theta)} +\delta_{3^{-k}\cos(-\pi/6 -\theta)} +\delta_{3^{-k}\cos(7\pi/6 -\theta)}].$$

We will now slightly modify $f$ for convenience. Note that

$$\hat{f}_{n,\theta}(x)=3^n\hat{\chi}_{[-3^{-n},3^{-n}]}(x)\cdot\prod_{k=1}^n\phi_\theta(3^{-k}x),$$

where $\phi_\theta(x)=\frac1{3}[e^{-icos(\theta-\pi/2)x} + e^{-icos(\theta-7\pi/6)x}+e^{-icos(\theta+\pi/6)x}]$. By factoring and changing the variable, we may instead write in place of $\phi_\theta$ the function

\begin{equation}\label{phit}
\varphi_t(x)=\frac1{3}[1+e^{-itx}+e^{-ix}],\,\,\, t\in [0,1]
\end{equation}

To do this, we split $[0,2\pi]$ into six cases: consider $\G_1$, which has three triangle centers. Under the projection map, a middle point migrates between the other two, either forward or in reverse. The change from $\theta$ to the parameterization $t$ corresponds to translating and rescaling the projections so that two projected triangles on the ends remain stationary during this migration of the middle triangle. In particular, we abolish $\theta$ and write $f_{n,t}$ from now on. We allow ourselves to drop the $t$ from functions and sets that depend on it when this dependence is not the pertinent feature in an argument.

For numbers $K,N>0$, define the following, (also depending on $t$ where appropriate):

\begin{equation}\label{fNstar}f_N^*(s):=\sup_{n\le N} f_{n,t}(s)\end{equation}
\begin{equation}\label{AKstar}A^*_K:=\{s:f_N^*(s)\geq K\}\end{equation}
\begin{equation}\label{Edef}E:=\lbrace t : |A^*_K|\leq \frac1{K^3}\rbrace\,.\end{equation}

$E$ is essentially the set of pathological $t$ such that $||f_{n,t}||_{L^2(s)}$ is small for all $n\leq N$, as in \cite{NPV}. In fact, we have this result, proved in Section $\ref{combi}$:
 
 \begin{theorem}
 \label{combin1}
 Let $t\in E$. Then
 $$
 \max_{n: 0\le n\le N} \|f_{n,t}\|^2_{L^2(s)}\le c\, K\,.
 $$
 \end{theorem}

The aim of Section $\ref{sec:fourier}$ is to prove the following:

 \begin{theorem}
 \label{mofE}
Let $\eps_0$ be a fixed small enough constant. Then for $N>>1$, $|E|<N^{-\eps_0}$.
 \end{theorem}
 
So let $K\approx N^{\eps_0}$, and suppose $|E|>\frac1{K}$. We will show that $N<N^*$, for some finite constant $N^*>>1$.

\subsection{{\bf Initial reductions}}\label{subsec:reductions}

Because of Theorem $\ref{combin1}$, we have $\forall t\in E$,
\begin{equation}
\label{intnuhat}
K\geq  ||f_{N,t}||^2_{L^2(s)}\approx ||\widehat{f_{N,t}}||^2_{L^2(x)}\geq C\int_1^{3^{N/2}}{|\widehat{\nu_N}(x)|^2dx}
\end{equation}

Let $m\approx 2\eps_0 \log\,N\approx 2\log\,K$. Split $[1,3^{N/2}]$ into $N/2$ pieces $[3^k,3^{k+1}]$ and take a sample integral of $|\widehat{\nu_N}|^2$ on a small block
\begin{equation}\label{defI}I:= [3^{n-m},3^n],\end{equation}
 with $n\in [N/4,N/2]$ chosen so that
\begin{equation}\label{choose n}
\frac{1}{|E|}\int_E{\int_{3^{n-m}}^{3^n}{|\widehat{\nu_N}(x)|^2dx\,dt}}\leq CKm/N\,.
\end{equation}
This choice is possible by $\eqref{intnuhat}$. Define 
 $$
 \tilde{E}:=\lbrace t\in E\setminus [1/2-3^{-m},1/2+3^{-m}]: \int_{3^{n-m}}^{3^n}{|\widehat{\nu_N}(x)|^2dx}\leq 2CKm/N\rbrace\,.
 $$
  It then follows that $|\tilde{E}|\geq\frac{1}{2K}$. We removed a small interval (of size $1/K^2$) around $1/2$ so that we may freely assume $|\varphi_t'(\lambda_j)|>c3^{-m}$ for all complex zeroes $\lambda_j$ of $\varphi$ having small enough imaginary part. It is an elementary consideration, but see also Section $\ref{complex}$.

Note that $\widehat{\nu_N}(x)=\prod_{k=1}^N{\varphi (3^{-k}x)}\approx\prod_{k=1}^n{\varphi (3^{-k}x)}$ for $x\in [3^{n-m},3^n]$.

So for $t\in E$,
$$\int_{3^{n-m}}^{3^n}{\prod_{k=1}^n{|\varphi_t(3^{-k}x)|^2}dx}\leq \frac{CKm}{N}\leq 2\eps_0N^{\eps_0-1}\log\,N.$$
Later, we will show that $\exists t\in E$ and absolute constant $A$ such that 

\begin{equation}\label{totalest}
\int_{3^{n-m}}^{3^n}{\prod_{k=1}^n{|\varphi_t(3^{-k}x)|^2}dx}\geq c3^{m-2\cdot Am}=cN^{2(1-2A)\epsilon_0}.
\end{equation}

The result: $2\eps_0 \log\,N\geq N^{1+(1-4A)\eps_0}$, i.e., $N\leq N^*$. In other words:

\begin{prop}
\label{reduction}
Inequality $\eqref{totalest}$ is sufficient to prove Theorem $\eqref{mofE}$. Further, inequality $\ref{totalest}$ can be deduced from Propositions $\ref{P1below}$ and $\ref{P1onSSV}$, as will be seen shortly.
\end{prop}

So let us prove inequality $\eqref{totalest}$.

First, let us write $\prod_{k=1}^n\varphi_t(3^{-k}x)=P_t(x)=P_{1,t}(x)P_{2,t}(y)$, where $P_2$ is the low frequency part, and $P_1$ is has medium and high frequencies:
$$P_{1,t}(x):=\prod^{n-m}_{k=1}\varphi_t(3^{-k}x)=\widehat{\nu_{n-m}}(x)$$ $$P_{2,t}(x)=\prod_{k=n-m}^{n}\varphi_t(3^{-k}x)=\widehat{\nu_m}(3^{m-n}x)$$
We want the following:

\begin{prop}\label{P1below}
Let $t\in E$ be fixed. Then $\int_{3^{n-m}}^{3^n}{|P_{1,t}(x)|^2dx}\geq C3^m$.
\end{prop}

We also want a proportion of the contribution to the integral separated away from the complex zeroes of $P_{2,t}$:

\begin{prop}\label{P1onSSV}
Let $\varepsilon^*$ be a small enough absolute constant to be seen in Section $\ref{sec:Analytic tiling}$, and let $SSV(t):=\{x\in I: |P_{2,t}(x)|\leq (\varepsilon^*/9)^m\}$. Then
$$\frac1{|\tilde{E}|}\int_{\tilde{E}}\int_{SSV(t)}|P_{1,t}(x)|^2dxdt\leq c3^m,$$
where $c$ is less than the $C$ from Proposition $\ref{P1below}$.
\end{prop}

$SSV(t)$ is so named because it is the $\textbf{set of small values}$ of $P_2$ on $I$. Note that while Proposition $\ref{P1below}$ will be proven for all $t\in\tilde{E}$, Proposition $\ref{P1onSSV}$ is an average. But from the average, one will be able to extract some $t\in\tilde{E}$ so that
$$\int_{SSV(t)}|P_{1,t}(x)|^2dx\leq c3^m,$$
and so combining this with Proposition $\ref{P1below},$
$$\int_{I\setminus SSV(t)}|P_{1,t}(x)|^2dx\geq c'3^m,$$

Thus Propositions $\ref{P1below}$ and $\ref{P1onSSV}$ suffice to prove Theorem $\ref{mofE}$, and Proposition $\ref{reduction}$ has been demonstrated.

Also, one may recall that $|\tilde{E}|\geq\frac1{2K}$, so that Proposition $\ref{P1onSSV}$ can be deduced from $$\int_{\tilde{E}}\int_{SSV(t)}|P_{1,t}(x)|^2dxdt\leq c\frac{3^m}{K}\approx 3^{m/2}.$$

First, let us fix $t\in E$ and prove Proposition $\ref{P1below}$ using Salem's trick on $$\int_0^{3^n}{|P_1(x)|^2dx}:$$

Let $h(x):=(1-|x|)\chi_{[-1,1]}(x)$, and note that $\hat{h}(\alpha)=C\frac{1-\cos\,\alpha}{\alpha ^2}>0$. Then if we write $P_1=3^{m-n-1}\sum_{j=0}^{3^{n-m}}{e^{i\alpha_jx}}$, we get
$$\int_0^{3^n}{|P_1(x)|^2dx}\geq 2\int_{-3^{n}}^{3^{n}}{h(3^{-n}x)|P_1(x)|^2dx}$$
$$\geq C(3^{m-n})^2[3^n\cdot 3^{n-m}+\sum_{j\neq k; j,k=1}^{3^{n-m}}3^n{\hat{h}(3^n(\alpha_j-\alpha_k))}]\geq C3^m.$$

To show that this is not concentrated on $[0,3^{n-m}]$, we will use Theorem \ref{combin1} and Lemma \ref{CETSQ}. We get
$$
\int_0^{3^{n-m}}{|P_1(x)|^2dx}= \int_0^{3^{n-m}}{|\widehat{\nu_{n-m}}(x)|^2dx}=3^{2(m-n)}\int_0^{3^{n-m}}|\sum_{j=0}^{n-m}e^{i\alpha_jx}|^2dx$$
$$\leq CK\leq C3^{\frac{m}{2}}.
$$

So now we have Proposition $\ref{P1below}$. The greater challenge will be Proposition $\ref{P1onSSV}$.

\subsection{Proposition $\ref{P1onSSV}$: The estimate on $\Pshp$.}
\label{subsec:P1onSSV}

Recall that $SSV(t):=\{ x\in I: |P_{2,t}(x)|\leq {(\varepsilon^*/9)^m} \}$.

To get Proposition $\ref{P1below}$, we will split $P_{1,t}$ into two parts, $\Pshp$ and $\Pflt$: a straightforward application of Lemma $\ref{CETSQ}$ to $\Pshp$ will get us part of the way there (for fixed $t$, the size of $SSV(t)$ does not overwhelm the average smallness of $\Pshp$), and the claims of Section $\ref{complex}$ applied to $\Pflt$ will further sharpen the final estimate to what we need.

Naturally, $\Pflt$ and $\Pshp$ are defined as the medium and high frequency parts of $P_{1,t}(x)$. Below, $\ell:= \alpha m$, for some large enough constant $\alpha$:
$$
\Pflt := \prod_{k=n-m-\ell}^{n-m-1}\varphi_t(3^{-k}x)=\widehat{\nu_{\ell-1}}(3^{m+\ell-n}x)\,,\,
$$
$$
\Pshp:= \prod_{k=1}^{n-m-\ell-1}\varphi_t(3^{-k}x)=\hat{\nu}_{n-m-\ell-1}(x).
$$

This is the claim of the subsection:
\begin{prop}\label{Pshpest}
$$\int_{SSV(t)}|\Pshp|^2dx\leq C''K3^m.$$
\end{prop}

We will see in Section $\ref{sec:Analytic tiling}$ that for each $t$, $SSV(t)$ is contained in $C\cdot 3^m$ neighborhoods of size $3^{n-m-\ell}$ around the complex zeroes $\lambda_j$ of $P_2$. (This is Corollary $\ref{SSV final}$, which sounds plausible because the highest frequency among all factors of $P_2$ is about $3^{m-n}$, and we are looking at an interval of length $<3^n$. But much care has to be taken to show that the zeroes do not resonate between factors.)

Fix $t$. Let
\begin{equation}\label{Ijt}I_j=[\lambda_j-3^{n-m-\ell},\lambda_j+3^{n-m-\ell}],\end{equation}
\begin{equation}\text{where }SSV(t)\subseteq\bigcup_jI_j\end{equation}
Choose $j$ for which $\int_{I_j}|\Pshp|^2dx$ is maximized. Then
$$\int_{SSV(t)}|\Pshp|^2dx\leq C3^m\int_{I_j}|\Pshp|^2dx\leq C3^m(3^{\ell+m-n})^2\int_{I_j}|\sum_{k=0}^{n-m-\ell} e^{i\alpha_jx}|^2.$$

Recall $|I_j|\leq 2\cdot 3^{n-m-\ell}$, so Lemma $\ref{CETSQ}$ and the definition of $E$ give us Proposition $\ref{Pshpest}$.

\subsection{{\bf Proposition $\ref{P1onSSV}$: The estimate on $\Pflt$}}
Of course we cannot just ignore $|\Pflt|^2$ in $\int_{\tilde{E}}\int_{SSV(t)}|\Pshp|^2|\Pflt|^2dx\, dt$, but one can bound it uniformly in each $I_j(t)$ by a Riesz product and then integrate in the $t$ variable. Because the shape of $\tilde{E}$ is rather complicated (see \cite{kenyon}), we will integrate our Riesz estimate on $|\Pflt|$ for all $t\in [0,1]\setminus (1/2-3^{-m},1/2+3^{-m})$, where now $x$ will be many functions $x_j(t)$ chosen to exhaust $SSV(t)$. When this is done, a factor of $K3^{m/2}$ will be cancelled out in the right-hand side of Proposition $\ref{Pshpest}$, finally proving Proposition $\ref{P1onSSV}$ and thus Theorem $\ref{mofE}$.

Now define
$$
r(x):=\frac{7+ 2\cos (x)}{9}\,,
$$
$$
R(x):=\prod_{k=n-m-\ell}^{n-m-1} r(3^{-k}\,x)
$$

The function $R(x)$ will estimate $|P_{1,t}^\flat(x_j(t))|$. The function $R$ is $2\pi\cdot 3^{n-m-1}$-periodic function. Note that its integral over a period is $2\pi\cdot (7/9)^{\ell}\cdot 3^{n-m-1}$. This is a general feature of Riesz products: when one integrates a full period, each factor can be identically replaced by its average. One can see this by changing the variable to get a $2\pi$-periodic Riesz product and using lacunarity of the frequencies to compute the $0$-th Fourier coefficient.

We will prove now
\begin{lemma}
\label{R1}
$|\varphi_t(x)|^2\leq \min (r(x), r(tx)\,).$ In particular,
$$|\Pflt|^2 \leq\min (R(x), R(tx)\,).$$
\end{lemma}

\begin{proof}
It is easy to prove something more general. Let $\alpha_j\in\R$, $j=1,2,...,P$. Then
$$|\sum_{j=1}^P e^{i\alpha x}|^2=P+2\sum_{1\leq j<j'\leq P}\cos((\alpha_j-\alpha_{j'})x)\leq (P^2-2)+2\cos((\alpha_{j_1}-\alpha_{j_2})x).$$

The lemma follows by letting $P=3$, $\alpha_{j_2}=0$ and choosing $\alpha_{j_1}$ and $\alpha_{j_3}$ from $1$ and $t$.
\end{proof}

 We will see shortly that we need both Riesz estimates. Each has an associated change of variables, and the pair is sufficiently ``separated away from simultaneous degeneracy."

 In Section $\ref{complex},$ we will have occasion to consider $SSV(t)$ as a subset of $[3^{n-m},3^n]\times [-3^{n-m},3^{n-m}]\subset\C$. We will see in Sections $\ref{complex}$ and $\ref{sec:Analytic tiling}$ that $$SSV(t)\subseteq\bigcup_{j=1}^J B(\lambda_j(t),3^{n-m-\ell}),$$
 where the $\lambda_j(t)$ are the complex zeroes of $P_{2,t}$. They are in fact simple, depending differentiably on $t$, and no more than $C3^{3m}$ of them have some contact with the big interval $I=[3^{n-m},3^n]$ (i.e., $J\leq C3^{3m}$). Each $\lambda_j(t)$ has its $t$ restricted to a time interval of size $c3^{-2m}$, which is called $D_r\cap\R$ (complex time $t\in D_r$ is considered in Section $\ref{complex}$). So we divide $SSV(t)\cap I$ into the intersections of the neighborhoods of these zeroes with the real interval $[3^{n-m},3^n]$ to get the intervals $I_j(t)$. This consideration is made for each $r=1,2,...,R\leq C3^{2m}$ separately, since the time neighborhood $D_r$ is of small enough size to control the number of zeroes $\lambda_j(t)$ entering and leaving the critical band $\R\times [-3^{-m},3^{-m}]$ during that time.

The $I_j(t)$ are centered at $x_j(t)=Re(\lambda_j(t))$, and have radius $3^{n-m-\ell}$. Lemma $\ref{uniform riesz}$ says that within each $I_j(t)$ with $t$ fixed, our Riesz estimates on $|\Pflt|^2$ are absolutely comparable independent of $x$, and the contants of comparability depend on nothing. Further, we will define $R_j^*(t):=R(x_j(t))\chi_{U_j}(t) + R(tx_j(t))\chi_{V_j}(t)$, where $U_j$ and $V_j$ are open and cover $D_r\cap\R$, and $x_j(t)$, $tx_j(t)$ are differentiable on $U_j$ and $V_j$ respectively, with derivative bounded above and below by constant multiples of $3^{k-m}$, where $k=n-m+1,...,n$ denotes the factor of $P_2$ such that $\varphi_t(3^{-k}\lambda_j(t))=0$. $U_j$ and $V_j$ each have at most $Cm$ components (See Section \ref{complex} for details).

Gathering all of this, consider a single $j$. Then 
$$\int_{D_r\cap\R}|P_{1,t}^\flat (x_j(t))|^2dt\leq C\int_{t_0}^{t_0+3^{-2m}}{R_j^*(t)dt}$$
$$\leq C3^{m-k} C'm\int_0^{3^{k-3m}}{R(t)}dt \leq C3^{m-k} C'm\int_0^{3^{n-m}}{R(t)}dt$$
$$\leq C'm3^{n-k}(\frac{7}{9})^{\ell}\leq C'm3^m(\frac{7}{9})^{\ell}.$$

This goes into the following, which uses Proposition $\ref{Pshpest}$:

$$\int_{D_r\cap\R}\int_{I_j(t)}|\Pshp|^2|\Pflt|^2dx\,dt\leq C\int_{D_r\cap\R}|P_{1,t}^\flat(x_j(t))|^2\int_{I_j(t)}|\Pshp|^2dxdt$$
$$\leq C'm3^m(\frac{7}{9})^{\ell} \cdot C''K3^m.$$

Summing over all $j$ to cover $SSV(t)$ and then summing over all $r$ to cover $[0,1]\setminus [1/2-3^{-m},1/2+3^{-m}]\supset\tilde{E}$,

$$\int_{\tilde{E}}\int_{SSV(t)}|P_{1,t}(x)|^2dx\leq 3^{Cm}CmK(\frac{7}{9})^{\ell}\leq 3^{m/2}$$

The last inequality is true (and perhaps much better, of course) once one chooses $\alpha$ large enough and lets $N$ (and therefore $m=2\epsilon_0\, \log\,N$ and $\ell=\alpha m$) be large.

This completes the proof of Proposition $\ref{P1onSSV}$ and of Theorem $\ref{mofE}$.

\bigskip

\section{Combinatorial part}
\label{combinatorics}

In this section, we show how Theorem $\ref{mainth1}$ follows from Theorem $\ref{mofE}$.

First, let us define
\begin{equation}\label{LthetaN}
\mathcal{L}_{\theta,N}:=proj_\theta\G_N.
\end{equation}

\begin{theorem}
\label{gooddir1}
Let $\beta>2$. (We used $\beta=3$ in the previous section). If $t\notin E$ (see definition $\eqref{Edef}$), then $|\mathcal{L}_{\theta, NK^{\beta}}| \le \frac{C}{K}.$
\end{theorem}

\begin{proof}
Let us use $\theta$ instead of $t$ and $x$ for the space variable on the non-Fourier side, since we do not use Fourier analysis in this proof. Fix $\theta$ and let $F:=A_K^*=\{x: f_N^*(x) \ge K\}$. We denote by $N_x$ the line orthogonal to direction $\theta$ and passing through $x$. We can call it needle at $x$. For every $x\in F$ there are at least $K$ triangles of size $3^{-r}, r=r(x), r\le N$, intersecting $N_x$. Mark them. Run over all $x\in F$. Consider all marked triangles. Consider all $3^{-N}$-triangles that are sub-triangles of marked ones. Call them ``green". Let $U$ be a family of green triangles.

We want to show
\begin{equation}
\label{star}
\text{card}\, U \ge c\cdot K\,|F|\, 3^N\,,
\end{equation}
\begin{equation}
\label{starstar}
|\text{proj}\,(\cup_{q\in U}q)| \le \frac{C}{ K}\,\text{card}\, U\,3^{-N}\,,
\end{equation}

Let $\phi:= \sum_{q\in U} \chi_q$. Then
$$
\int\phi\,dx = \text{card}\, U\, 3^{-N}\,.
$$
Let $M$ denote  uncentered maximal function. To prove \eqref{starstar} it is enough to show that
$$
q\in U \Rightarrow  \text{proj}\,q\subset \{x: M\phi(x) >\frac{K}{C}\}\,,
$$
and then to use Hardy--Littlewood maximal theorem. But to prove this claim is easy. In fact, let $x\in \text{proj}\,q, q\in U$, then there exists $Q$--the maximal (by inclusion) marked triangle containing $q$. Consider $I:= [x-10\,\ell(Q), x+10\,\ell(Q)]$. This segment contains the projections of at least $K$ disjoint triangles $Q_1:=Q, Q_2,..., Q_K,...$, of the same sidelength, which intersect $N_{x_0}$, where $x_0$ is a point because of which $Q=Q_1$ was marked. (The reader should see that $x_0$ lies really well inside $I$.) So $I$ contains the projections of at least $\frac{\ell(Q)}{\ell(q)}\cdot K$ green traingles. Whence,
$$
\int_I\phi\,dx \ge \ell(q)\cdot \frac{\ell(Q)}{\ell(q)}\cdot K\ge \frac1{20} |I|\,K\,.
$$So
$$
M\phi(x)> \frac1{20}\,K\,.
$$
We proved \eqref{starstar}.

\medskip

Also we proved that $F\subset \{ x: M\phi (x) \ge \frac{K}{20}\}$. Therefore, by Hardy--Littlewood maximal theorem
$$
|F|\le |\{ x: M\phi (x) \ge \frac{K}{20}\}|\le \frac{C\,\int\phi}{K} = C\,\text{card}\,U\, 3^{-N}\,K^{-1}\,.
$$
This is \eqref{star}.

\bigskip

Let us estimate $|\mathcal{L}_{\theta, N\, K^{\alpha}}|$ using \eqref{star} and \eqref{starstar}.
The first step:
$$
|\mathcal{L}_{\theta, N}| \le |\text{proj}\,(\cup_{q\in U} q)| + 3^{-N} (3^N -\text{card}\, U)\le
$$
$$
\frac{C}{K}\text{card}\, U \,3^{-N} + (3^N-\text{card}\,U) 3^{-N}\,.
$$
We do not touch the first term, but we improve the second term by using self-similar structure and going to step $2N$ (inside traingles which are not green there are ``green" triangles of size $3^{-2N}$). They are just self-similar copies of the original green triangles.
 Then
we have the second step:
$$
|\mathcal{L}_{\theta, N}| \le \frac{C}{K}\text{card}\, U \,3^{-N} + \text{the rest}\le
$$
$$
\frac{C}{K}\text{card}\, U \,3^{-N} + (3^N-\text{card}\,U) \frac{C}{K}\text{card}\, U \,3^{-2N} +(3^N-\text{card}\,U)^2\,3^{-2N}\,.
$$

Now we leave first two terms alone and having $(3^N-\text{card}\,U)^2$ traingles of size $3^{-2N}$ we find again ``green" triangles inside each of those, now green traingles of size $3^{-3N}$. They are just self-similar copies of original green triangles.

 Then
we have the third step:
$$
|\mathcal{L}_{\theta, 3N}| \le 
\frac{C}{K}\text{card}\, U \,3^{-N} + (3^N-\text{card}\,U) \frac{C}{K}\text{card}\, U \,3^{-2N} +\text{the rest} \le
$$
$$
\frac{C}{K}\text{card}\, U \,3^{-N} + (3^N-\text{card}\,U) \frac{C}{K}\text{card}\, U \,3^{-2N}+ (3^N-\text{card}\,U)^2 \frac{C}{K}\text{card}\, U \,3^{-2N}+
$$
$$
(3^N-\text{card}\,U)^3\,3^{-3N}\,.
$$
After the $l$-th step:
$$
|\mathcal{L}_{\theta, l\,N}| \le \frac{C}{K}\text{card}\, U \,3^{-N} (1 +(3^N-\text{card}\,U)3^{-N}+...
$$
$$
+(3^N-\text{card}\,U)^{l-1}3^{-(l-1)N}) + (3^N-\text{card}\,U)^{l}3^{-(lN})\,.
$$
So
$$
|\mathcal{L}_{\theta, l\,N}| \le \frac{C}{K}\text{card}\, U \,3^{-N}  \frac{(1-(1-\frac{\text{card}\,U}{3^N})^l)}{(1-(1-\frac{\text{card}\,U}{3^N})} +
$$
$$
 e^{-\frac{\text{card}\,U}{3^N} l}=: I +II\,.
$$
Notice that by \eqref{star} $II\le e^{-K|F|l} \le e^{-K}$ if the step $l$ is chosen to be $l=1/|F|\le K^{\beta}$. However, we always have $I\le \frac{C}{K}$. So Theorem \ref{gooddir1} is completely proved.
\end{proof}

From Theorems $\ref{mofE}$ and $\ref{gooddir1}$, it is not hard to get Theorem $\ref{mainth1}$.

\section{The complex analytic part}
\label{complex}

\subsection{Elementary facts about $\varphi$}
\label{entire}

In this section, we investigate the various nice properties of $\varphi$, considered as a function of the complex variable $z=x+iy$, with $x>0$. We will work mostly with
$$\phitil(z):=3\varphi_t(z)=1+e^{-itz}+e^{-iz}.$$
Recall that $t\in [0,1]$ earlier. We also complexify $t$: $t=u+iv\in T$, where $T:=[0-3^{-m},1+3^{-m}]\times[-3^{-m},3^{-m}]$.
Define also 
\begin{equation}\label{Ttilde}\tilde{T}:=T\setminus([1/2-3^{-m},1/2+3^{-m}]\times[-3^{-m},3^{-m}]).\end{equation}
This deletion is motivated by Section $\ref{branch1}$.

Note that we are trying to control the zeroes of
$$P_{2,t}(z)=\prod_{k=n-m+1}^{n}\varphi_t(3^{-k}z)=\widehat{\nu_{m-1}}(3^{m-1-n}z),$$

for $Re(z)\in [3^{n-m},3^n]$. For this purpose, it suffices to consider $\phitil$ restricted to not far from
\begin{equation}\label{Itilde}z\in\tilde{I}:=[3^{-m},3^m].\end{equation}
Notice that at the end, we will have to multiply by $3^k$ to get the location of the zeroes back to where they belong in the big picture of Section $\ref{sec:fourier}$. Call the zeroes of $\phitil$ by the name $\lambdatil$, and only call the zeroes $\lambda$ when they are regarded as zeroes of the factors of $\varphi(3^{-k}\cdot)$.

To use Blaschke estimates along the real line, we need for fixed $t$ that $x$ is never far from $z$ such that $|\phitil(z)|>1/2$.

\begin{lemma}\label{H}
There exists $H>0$ such that $\forall t\in T,x>0$
$$max_{|z-x|\leq H}|\phitil(z)|\geq 1/2.$$
\begin{proof}
First, consider $t$ real. Notice that for $t\leq 1/2$, $e^{-iz}$ is the dominant summand for $y\geq H$, and for $t\geq 1/2$, $1$ is dominant for $y\leq -H$. Thus we are never more than the distance $H$ from a point $z$ at which $|\phitil|>1/2$.

For $t$ complex, we can write $\phitil(z)=1+c_1e^{vx+uy}+c_2e^{y}$, for some $|c_1|=|c_2|=1$ depending on $z,t$. Since $x\in\tilde{I}$ and $t\in\tilde{T}$, $xv\leq C$. By choosing a larger value of $H$ if needed, we can make either $1$ or $c_2e^y$ dominant like before.
\end{proof}
\end{lemma}

\begin{lemma}\label{zeroes}
There exists an absolute constant M such that in $B(x,1)$, $\phitil$ has at most $M$ complex zeroes $\lambdatil_j$. Further, the set of $z$ where $|\phitil(z)|\leq \varepsilon<<1$ is contained in 
$$\bigcup_j B(\lambdatil_j,C\varepsilon^{1/M}).$$

\begin{proof}
To use Lemma $\ref{schke1}$, we need the previous lemma, and we need to bound $|\phitil(z)|$ from above on a neighborhood of the point $z_0$ where $|\phitil(z_0)|>1/2$. This neigborhood needs to contain $B(x,1)$. This is not hard, either: $vx\leq C$, and $|y|\leq H$ at the the point $z_0$, so $|\phitil(z)|$ is easily bounded absolutely in a $2H$-neighborhood of $z_0$. The other claim is Lemma $\ref{schke2}$.
\end{proof}
\end{lemma}

\subsection{Branch points of $\phitil$ and analytic continuation of zeroes.}
\label{branch1}

For $t\in T$ and $z\in\R\times [-H,H]$, we call the pair $(t,z)$ a {\it branch point} of $\phitil$ if

\begin{equation}
\label{eqbranch}
\begin{cases}
\phitil(z) =0 \\ 
\frac{\partial}{\partial z} \phitil(z) =0\,.
\end{cases}
\end{equation}

\begin{lemma}
\label{realbrp}
There are no branch points such that $t$ is a real number in $[0,1]$.

\begin{proof}

If \eqref{eqbranch} is valid then
\begin{equation}
\label{eqbranch1}
\begin{cases}
e^{-itz} = -1 - e^{-iz}\\ 
te^{-itz}+e^{-iz}=0\,.
\end{cases}
\end{equation}

Hence $e^{-iz}(1-t) - t=0$.
Of course $t=0,1$ are impossible. So $e^{-iz}=\frac{t}{1-t}$.
Doing the other substitution, one gets $e^{-itz}=\frac1{t-1}$.

Taking absolute values:
\begin{equation}
\label{eqbranch2}
\begin{cases}
e^y = \frac{t}{1-t}\\ 
e^{ty} = \frac1{1-t}
\end{cases}
\end{equation}
$0<t<1$, so $e^y<e^{ty}\Rightarrow y<0$.

But then $\frac1{1-t}=e^{ty}<1$, a contradiction.
\end{proof}
\end{lemma}

This allows us to analytically continue zeroes: $\lambdatil_j(t)$, holomorphic in $t$ on some neighborhood $T_j$ of $[0,1]$, satisfying $\phitil(\lambdatil_j)=0$. But we would like to control $|\lambda'(t)|$, so we restrict to $\tilde{T}$ (\footnote{Definition $\eqref{Ttilde}$}), where estimates are easier to come by. Note that if $z$ is a zero of $\phitil$, $Im(z)\leq 2\cdot3^{-m}$, and $t\in\tilde{T}\cap\R$, we have
\begin{equation}\label{phitilprime}
3\geq |\phitil'(z)|=|(1-t)e^{itz}-t|\geq c|1-2t|\geq c3^{-m}.
\end{equation}

\begin{lemma}\label{holo-DE}
For this lemma, the subscripts $z$ and $\tau$ will denote partial derivatives. Let $\Phi(z,\tau)$ be holomorphic in both complex variables. Let $\Phi(z_0,\tau_0)=0$ and $\Phi_z(z_0,\tau_0)\neq 0$. Suppose that on some neigborhood $B(z_0,\delta_1)\times B(\tau_0,\delta_2)$, one has $$|\Phi_z(z,\tau)|\geq\eta\,\,\,\text{and}\,\,\,|\Phi_\tau(z,\tau)|\leq C.$$
Then there exists a unique holomorphic function $\lambdatil:B(w_0,\varepsilon)\to\C$ such that
$$\begin{cases}\lambdatil(\tau_0)=z_0\\
\Phi(\lambdatil(w),w)=0.
\end{cases}$$
$$Also,\,\,\,\lambdatil'(\tau)=-\frac{\Phi_\tau(\lambdatil(\tau),\tau)}{\Phi_z(\lambdatil(\tau),\tau)}\,\,\,\text{ and }\varepsilon =min\{c\delta_1\eta,\delta_2\}.$$
\end{lemma}

Applied to our case, Lemma $\ref{holo-DE}$ and $\eqref{phitilprime}$ gives us this estimate for $\lambdatil_j'(t)$, when $|Im(\lambdatil_j(t))|\leq 2\cdot3^{-m}$:
\begin{equation}\label{lambda prime}
|\lambdatil_j'(t)|\leq C3^m
\end{equation}

This is still somewhat fast, but we will see later in Section $\ref{sec:Analytic tiling}$ that we only need be concerned when $|Im(\lambdatil_j(t))|\leq 3^{-m}$. In this case, we say $\lambda_j(t)$ is in the critical band $\tilde{Q}$ of radius $3^{-m}$ around $\tilde{I}$(\footnote{Definition $\eqref{Itilde}$}). Consider also the band $\tilde{R}$ of radius $2\cdot 3^{-m}$ around $\tilde{I}$. The factor $2$ establishes a buffer through which it takes at least $c3^{-2m}$ ``seconds" to pass before entering the critical band from the outside. So if we count zeros in $\tilde{R}$ at an initial real time $t_0$, only those zeroes can enter the critical band $\tilde{Q}$ during this small interval of complex time. This is the content of the next lemma.

Let us state this as a lemma. Cover $\tilde{T}\cap\R$ by discs
\begin{equation}\label{Dr}
D_r:=B(t_r,c3^{-2m}),\,\,\,t_r:=cr3^{-2m},\,\,\,r=1,2,...,R.
\end{equation}

\begin{lemma}\label{count lambda}
Fix an $r$. Let $(z_0,t_0)\in\tilde{Q}\times (D_r\cap\R)$ be a pair such that $\tilde{\varphi}_{t_0}(z_0)=0$. Then all such pairs belong to a union of paths $\lambdatil_j(t)$, for $j=1,2,...,J$, where $J\leq C3^{m}$.
\begin{proof}
For any such pair $(z_0,t_0)$, one can analytically continue $z_0$ as a function $\lambdatil(t)$ in a disc of radius $c3^{-2m}$ around $t_0$ (this is Lemma $\ref{holo-DE}$). Such a disc meets $t_r$. So $z_0=\lambdatil(t_0)$, for some $\lambdatil:D_r\to\C$.

$\lambdatil_j(t_r)\in\tilde{R}$ because of $\eqref{lambda prime}$ and because $\lambdatil_j(t_0)\in\tilde{Q}$. Because of Lemma $\ref{zeroes}$, at the initial time $t_r$ there were at most $C3^m$ such $\lambdatil_j$ within distance $1$ of $\tilde{I}\subseteq\tilde{R}$. So at time $t_r$, we can number the zeroes $\lambdatil_j(t)\in\tilde{R}$, $j=1,2,...,J$, and for each $t\in\D_r$, all zeroes of $\tilde{\varphi_t}$ in $\tilde{Q}$ lie along one of these paths.
\end{proof}
\end{lemma}

\subsection{Holomorphic extension of the real parts of the $\lambdatil_j$}
\label{xj}

Let 
$$\tilde{x}_j(t):=\frac1{2}(\lambdatil_j(t)+\bar{\lambdatil}_j(\bar{t})).$$
Then the $\tilde{x}_j$ are holomorphic, and for $t$ real, $\tilde{x}_j(t)=Re(\lambdatil_j(t))$.
We use $\tilde{x}(t)$ here, analogous to the remark about $\lambdatil(t)$. We will remove the tilde when we change variables by $3^k$ and adapt the zeroes back to the factors of $P_2$.

In Section $\ref{sec:fourier}$, we consider functions 
\begin{equation}\label{gi def}
g_{1,j}(t)=\tilde{x}_j(t)\text{ and }g_{2,j}(t)=t\tilde{x}_j(t).
\end{equation}
These are the changes of variable in the Riesz estimates. For fixed $j$, we have $\eqref{lambda prime}$, so
\begin{equation}\label{size g}
|g_{i,j}|,|g_{i,j}'|\leq C3^{m}.
\end{equation}

We sometimes drop the $j$ when it is considered fixed in a context.

\begin{lemma}\label{g dichot}
Fix $i$ and $j$. Within $D_r$ on which $g_i$ is defined, one of the following is true:
\begin{equation}\label{small for all}
||g_i'||_\infty\leq C3^{-m}\end{equation}

\begin{equation}\label{not many zeroes}
\#\{ t:g_i'=0\}\leq Cm.
\end{equation}
\begin{proof}
Suppose $\eqref{small for all}$ is false. Divide $g_i'$ by $C3^{-m}$, so that Lemma $\ref{schke1}$ applies to this new function. Then the conclusion is exactly $\eqref{not many zeroes}$.
\end{proof}
\end{lemma}

This is good because of the following:

\begin{lemma}\label{g zeroes}
For fixed $j$, it is impossible for $|g_i'(t)|\leq C3^{-m}$ to happen for $i=1,2$ simultaneously at $t$. In particular, $(\ref{small for all})$ cannot happen for $i=1,2$ simultaneously.
\begin{proof}
Fix $j$ and fix $t\in\tilde{T}$. Suppose $|g_1'(t)|,|g_2'(t)|<\varepsilon=C3^{-m}$.
By direct computation from the definition $(\ref{gi def})$, one gets $|\tilde{x}_j(t)|=|g_2'(t)-tg_1'(t)|\leq 2\varepsilon$.
But $|\tilde{x}_j(t)|\geq c3^{-m}$. It follows that $\varepsilon\geq c'3^{-m}$.
\end{proof}
\end{lemma}

\begin{cor}\label{other g good}
Fix $j$, and let $\lambdatil_j$ be defined on $D_r$. For at least one of $i=1,2$, $(\ref{not many zeroes})$ holds. At each such zero $t$ of $g_i'$, $|g_{3-i}'(t)|\geq c3^{-m}$.
\end{cor}

We need a little more. First, notice that for real $t$, the $g_{i,j}(t)$ are real, as are the $g_{i,j}'(t)$.

\begin{lemma} For each $j$, each $D_r\cap\R$ can be covered by real open sets $\tilde{U}_j$, $\tilde{V}_j$ so that $|g_{1,j}'(t)|\geq c3^{-m}$ on $\tilde{U}_j$ and $|g_{2,j}'(t)|\geq c3^{-m}$ on $\tilde{V}_j$. Further, $\tilde{U}_j$ and $\tilde{V}_j$ are unions of at most $Cm$ open intervals, where $C$ does not depend on anything.
\begin{proof}
In all of $D_r$, consider $p=g_1+g_2$ and $m=g_1-g_2$, and imitate the last two lemmas.

$|p'+m'|=2|g_1'|$ and $|p'-m'|=2|g_2'|$, so it is impossible for $|p'|,|q'|\leq c3^{-m}$ simultaneously. Now we have cases:
$$\begin{cases}||p'||_\infty\leq c3^{-m}\\ ||m'||_\infty\leq c3^{-m}\\ \#\{t\in D_r:p'=0\,\,or\,\,m'=0\}\leq Cm\end{cases}$$

The third case exhausts the remaining possibilities exactly as in Lemma $\ref{not many zeroes}$. Note that in either of the first two cases, $||g_1'|-|g_2'||\leq \frac{c}{2}3^{-m}$, so by Corollary $\ref{other g good}$, $|g_1'|,|g_2'|\geq \frac{c}{2}3^{-m}$ throughout $D_r$. So by perhaps changing the constant $c$, these cases are settled. ($\tilde{U}$ and/or $\tilde{V}$ may be taken to be $D_r\cap\R$.)

In the last case, we now restrict the above complex analytic information to the real line, and remember that the $g_i'$ are real. In particular, there are only $Cm$ such $t\in D_r\cap\R$. Away from such $t$, we are in an interval where either $|g_1'|>|g_2'|$ or the opposite. But by Lemma $\ref{other g good}$, the larger of the two is always larger than $c3^{-m}$, and so the interval is a component of $\tilde{U}$ or $\tilde{V}$, accordingly.

\end{proof}
\end{lemma}

\subsection{{\bf Rescaling back, and uniform Riesz bounds}}

Let $\lambda_j(t)=3^k\lambdatil_j(t)$, $x_j(t)=3^k\tilde{x}_j(t)$, etc. Now everything moves $3^k$ times as fast and has neighborhoods $3^k$ times as large, and possibly shows up in the interval $I=[3^{n-m},3^n]$ once for each $k=n-m+1,...,n$. Tildes can be removed from everything in this way, and a copy gets plugged into Section $\ref{sec:fourier}$ for each such $k$. So now we will know how to integrate the function $R_j^*(t)=R_j(x_j(t))\chi_{U_j}(t)+R_j(tx_j(t))\chi_{V_j}(t)$:

\begin{lemma}\label{uniform riesz}
For all $t\in D_r\cap\R$, for each of its $j$, and for each $x\in I_j(t)$, one has
$$R_j^*(t)\leq C[R_j(x)\chi_{U_j}(t)+R_j(tx)\chi_{V_j}(t)]$$
\begin{proof}
Recall:
$$R(x):=\prod_{k=n-m-\ell}^{n-m-1}r(3^{-k}x),\,\,r(x):=\frac{7+2\cos\,(x)}{9},$$
and $I_j(t):=[x_j(t)-3^{n-m-\ell},x_j(t)+3^{n-m-\ell}].$

The $\chi_{U_j},\,\,\chi_{V_j}$ truncate the small values out of our considerations. In each Riesz product, each factor belongs to $[\frac{5}{9},1]$, and in fact one could let $k\to +\infty$ in the above product and get geometric convergence, uniform in $x$ on the given interval. So one gets a constant like $(\frac{9}{5})^{C}$.
\end{proof}
\end{lemma}

The next section explains why considering the zeroes of the different factors of $P_2$ separately does no real harm to the main argument.

\section{{Analytic tiling}}\label{sec:Analytic tiling}
\subsection{{\bf Preamble}}\label{subsection:preamble}
In Section $\ref{complex}$, only a single factor $\phitil$ was considered at a time. In this part, we show that the product
$$\Phi_m(z)=\prod_{k=1}^m\phitil(3^{-k}z)=3^{m+1}P_2(3^{n-m}z)$$ has a well-behaved $\textbf{set of small values}$: only one factor may be critically small at a given time and place, and the $\textbf{product}$ of the remaining terms is no smaller than $3^{-m}$, so one can estimate integrals along $SSV(t)$ by considering the zeroes of each factor separately.

Something much worse could have happened: all or most terms could have been less than $3^{-m}$ simultaneously, so that outside of the neighborhood, one would only have the esitmate $\Phi_m(z)\geq (3^{-m})^m=3^{-m^2}$. Our set of small values would have been a set of very small values indeed, resulting in the weaker final estimate $Fav(\G_n)\leq e^{-c\sqrt{log\,n}}$. See also the discussion, Section $\ref{discu}$.

\subsection{{\bf Result}}\label{subsection:result}

We will regard $t\in [0,1]$ as fixed here, reclaiming the subscript for other purposes.
First, some definitions. $c>1$ will be an absolute constant. $\delta$ will be a small enough absolute constant, $x_0\in\R$, $m\in\mathbb{N}$ is large.
$R:=[x_0-\delta,x_0+\delta]\times[-\delta,\delta]$, and $2R:=[x_0-2\delta,x_0+2\delta]\times[-2\delta,2\delta]$
$$\tilde{\varphi} (z):=1+e^{-iz}+e^{-itz}$$
$$\tilde{\varphi}_k(z):=\tilde{\varphi}(3^{-k}z)$$
$$\Phi(z):=\prod_{k=0}^m\tilde{\varphi}_k(z)$$
Also of interest will be another function, $\Phi_{k_0}:=\Phi/{\tilde{\varphi}_{k_0}}$.
The most important thing to prove, and the place where so-called analytic tiling comes into play, is in the proof of

\begin{prop}\label{Phik0}
$\forall z$ with $|Im(z)|\leq\delta$, $\max_{k_0}|\Phi_{k_0}(z)|\geq 3^{-m}$.
Further, say that $k_0$ is $\textbf{critical}$ if $\exists z_0\in R$ with $|\tilde{\varphi}_{k_0}(z_0)|<3^{-m}$. Then a critical $k_0$ is unique whenever it exists.
\end{prop}

This will lead to the following:

\begin{prop}\label{SSVR}
Let $M$, $c$ be sufficiently large absolute constants.
Let $$SSV(R):=\{z\in R:|\Phi(z)|<(\varepsilon^*/3)^m\}$$
If there is a critical $k_0$, let
$$SSV_{k_0}(R):=\{z\in R:|\tilde{\varphi}_{k_0}(z)|<{\varepsilon^*}^m\}.$$
Then $SSV(R)=\emptyset$ if there is no critical $k_0$, and otherwise
$$SSV(R)\subseteq SSV_{k_0}(R)\subseteq\bigcup_jB(\lambda_j,C3^m{\varepsilon^*}^{m/M}),$$
where the $\lambda_j$ are the zeroes of $\tilde{\varphi}_{k_0}$ in $2R$(\footnote{Here $2R$ is concentric with $R$}).
\begin{proof}
If there is no critical $k_0$, then for each $z\in R$ there is some $k_0$ such that $|\Phi(z)|=|\Phi_{k_0}(z)\tilde{\varphi}_{k_0}(z)|\geq (\varepsilon^*/3)^m$. Thus
$SSV(R)=\emptyset$.

If there is a critical $k_0$, then $|\Phi|=|\Phi_{k_0}\tilde{\varphi}_{k_0}|\geq 3^{-m}|\tilde{\varphi}_{k_0}|$ shows that $SSV(R)\subseteq SSV_{k_0}(R).$
By the Blaschke estimate (Lemma $\ref{zeroes}$), $\tilde{\varphi}_{k_0}$ has at most $M$ zeroes $\lambda_j$ in $B(x_0,3^k)\subseteq B(x_0,3^m)$, and $$SSV_{k_0}(R)\subseteq\bigcup_jB(\lambda_j,C3^{k_0}(\varepsilon^*)^{m/M})\subseteq\bigcup_jB(\lambda_j,C3^m(\varepsilon^*)^{m/M}).$$
\end{proof}
\end{prop}

\begin{cor}
Let $C(3{\varepsilon^*}^{1/M})^m<3^{-\ell}=3^{-\alpha m}$, i.e., $\varepsilon^*<c3^{-M\alpha}$. Then the neighborhoods of small values have diameter $<3^{-\ell}$, and there are no more than $2M/\delta$ of them per unit interval.
\end{cor}

\begin{cor}\label{SSV final}
In the setting of Section $\ref{sec:fourier}$, this says that $SSV(t)$ is contained in $C3^m$ intervals of size $3^{n-m-\ell}$. This is by changing variables and by going back to $\varphi$ instead of $\tilde{\varphi}$ by multiplying $3^{n-m}$ back in.
\end{cor}

The main idea behind Proposition $\ref{Phik0}$ is to analyze the stability under perturbations of the solution to the following equations, unique up to swapping $w_1$ with $w_2$:

$\begin{cases}|w_1|=|w_2|=1\\1+w_1+w_2=0\end{cases}$

Clearly $w_j=e^{2\pi ij/3}$. What is interesting about this is that $1+(w_1)^{3^k}+(w_2)^{3^k}=3$ $\forall k=1,2,3,...$. This is stable, if use $k$ to control the size of the perturbations of the $w_j$.

\begin{lemma}\label{stable}
Let $|y_1|,|y_2|\leq c3^{-k'}$, and suppose that $w_1,w_2$ satisfy:
$$\begin{cases}|w_j|=e^{y_j}\\1+w_1+w_2=0\end{cases}$$
Then for $1\leq k\leq k'+1$, $|1+(w_1)^{3^k}+(w_2)^{3^k}|\geq 2.\\$
(In fact, $Re(1+(w_1)^{3^k}+(w_2)^{3^k})\geq 2$)
\begin{proof}
Write $w_j=e^{ix_j+y_j}$. Without loss of generality, $|x_j-2\pi j/3|\leq C3^{-k'}$. Thus $(w_j)^{3^k}$ have the appropriate arguments and magnitudes.
\end{proof}
\end{lemma}

\begin{cor}\label{phistable}
If $|\tilde{\varphi}_{k'}(z)|< c\delta 3^{-k^*}$, then $\forall k=k'-k^*,...,k'-1,$ one has $|\tilde{\varphi}_{k}(z)|\geq 2$
\begin{proof}
Let $w_1=e^{i3^{-k'}z}$, $w_2=e^{it3^{-k'}z}.$
\end{proof}
\end{cor}

Finally, let us prove Proposition \ref{Phik0}.

By induction. $m=0$ is clear. Assume Proposition \ref{Phik0} for $m-1$.
Fix $z$ in $R$. If $|\tilde{\varphi}_m(z)|\leq 3^{-m}$, then $k_0=m$, since all other factors must be at least 2 due to Corollary $\ref{phistable}$.
Now let $3^{-j-1}\leq |\tilde{\varphi}_m(z)|\leq 3^{-j}$, for some $j<m$ (or just induct if $|\tilde{\varphi}_m(z)|>1$).
Then $|\tilde{\varphi}_{m-k}(z)|\geq 2$ for all $k=1,...,j$, again by Corollary $\ref{phistable}$.
Thus $$|\prod_{k=m-j}^m\tilde{\varphi}_k(z)|\geq 3^{-j-1}.$$
By the induction hypothesis,
$$\exists k_0:|\prod_{k=0,k\neq k_0}^{m-j-1}\tilde{\varphi}_k(z)|\geq 3^{-m+j+1}$$
These two inequalities yield $|\Phi_{k_0}|\geq 3^{-m}$.

Next, we show that there can be at most one critical $k_0$. If there is a critical $k_0$, consider the largest. This means that $\exists z_0\in SSV_{k_0}(R)$. So $z_0$ lies in a small neighborhood of a zero $\lambda$ of $\tilde{\varphi}_{k_0}$ in $2R$ (concentric). Since $|\tilde{\varphi}_{k_0}'|\leq 2\cdot 3^{-k_0}$, it follows that $|\tilde{\varphi}_{k_0}|\leq C\delta 3^{-k_0}$ in $R$. Thus $|\tilde{\varphi}_k(z)|\geq 2$ for all $k<k_0$ and for all $z\in R$. $k_0$ was chosen to be the largest, so it is unique. $\square$

\section{Some important standard lemmas}
\label{lemmas}
There are a few important lemmas which we have appealled to repeatedly. The first claim, Lemma \ref{CET}, uses the Carleson imbedding theorem. A stronger version, Lemma $\ref{CETSQ}$, uses general $H^2$ theory. Its importance lies in its ability to establish a key relationship between the level sets of $f_{n,t}$ and the $L^2$ norm of $\widehat{f_{n,t}}$. This is because the Fourier transform changes the centers of intervals into the frequencies of an exponential polynomial.

The second claim we split into Lemmas $\ref{schke1}$ and $\ref{schke2}$. Given a bounded holomorphic function on the disc, its supremum, and an interior non-zero value, these lemmas bound the number of zeroes and contain the set of small values within certain neighborhoods of these zeroes.

\subsection{{\bf In the spirit of the Carleson imbedding theorem}}
\begin{lemma}
\label{CET}
Let $j=1,2,...k$, $c_j\in\C$, $|c_j|=1$, and $\alpha_j\in\R$. Let $A:=\lbrace \alpha_j\rbrace_{j=1}^k$. Then
$$
\int_0^1{|\sum_{j=1}^k{c_je^{i\alpha_jy}}|^2dy}\leq C\,k\cdot \sup_{I\text{ a unit interval}}\card \{A\bigcap I\}\,.
$$
\end{lemma}

\begin{proof}
Let $A_1:=\{\mu= \alpha +i: \alpha\in A\}$. Let $\nu:=\sum_{\mu\in A_1} \delta_{\mu}$. This is a measure in $\C_+$. Obviously its Carleson constant
$$
\|\nu\|_C :=\sup_{J\subset \R,\, J\,\text{is an interval}}\frac{\nu(J\times [0,|J|])}{|J|}
$$ can be estimated as follows
\begin{equation}
\label{CETeq}
\|\nu\|_C \le 2\,\sup_{I\text{ a unit interval}}\card \{A\bigcap I\}\,.
\end{equation}

Recall that
\begin{equation}
\label{CETeq2}
\forall f\in H^2(\C_+)\,\,\int_{C_+} |f(z)|^2 \,d\nu(z) \le C_0\, \|\nu\|_C\|f\|_{H^2}^2\,,
\end{equation}
where $C_0$ is an absolute constant.
Now we compute
$$
\int_0^1{|\sum_{j=1}^k{c_je^{i\alpha_jy}}|^2dy}\leq e^2\int_0^1{|\sum_{j=1}^k{c_je^{i(\alpha_j+i)y}}|^2dy}\leq
$$
$$
 e^2\int_0^{\infty}{|\sum_{j=1}^k{c_je^{i(\alpha_j+i)y}}|^2dy} = e^2\int_{\R}|\sum_{\mu\in A_1}\frac{c_{\mu}}{x-\mu}|^2\,,
 $$
 where $c_{\mu} := c_j$ for $\mu= \alpha_j +i$. The last equality is by Plancherel's theorem.
 
 We continue
$$
\int_{\R}|\sum_{\mu\in A_1}\frac{c_{\mu}}{x-\mu}|^2 =\sup_{f\in H^2(C_+),\, \|f\|_2\le 1}\bigg|\langle f,\sum_{\mu\in A_1}\frac{c_{\mu}}{x-\mu}\rangle\bigg|^2=
$$
$$
 4\pi^2\sup_{f\in H^2(C_+),\, \|f\|_2\le 1}|\sum_{\mu\in A_1}c_{\mu}f (\mu)|^2\le C\,\card\{A_1\}\sup_{f\in H^2(C_+),\, \|f\|_2\le 1}\sum_{\mu\in A_1} |f(\mu)|^2 \le
 $$
 $$
 C\,\card\{A\}\sup_{f\in H^2(C_+),\, \|f\|_2\le 1}\int_{C_+}|f(z)|^2\,d\nu(z)\le 
 2C_0C\,\card\{A\}\,\sup_{I\text{ a unit interval}}\card \{A\bigcap I\}\,.
 $$
 This is by \eqref{CETeq1} and \eqref{CETeq}. The lemma is proved.

\end{proof}

\medskip

Now we are going to prove a stonger assertion by a simpler approach. This stronger assertion is what is used in the main part of the article.

\medskip

\begin{lemma}
\label{CETSQ}
Let $j=1,2,...k$, $c_j\in\C$, $|c_j|=1$, and $\alpha_j\in\R$. Let $A:=\lbrace \alpha_j\rbrace_{j=1}^k$. Then
Suppose
\begin{equation}
\label{sumch}
\int_{\R} (\sum_{\alpha\in A} \chi_{[\alpha-1, \alpha+1]}(x))^2\,dx \le S\,,
\end{equation}
Then there exists an abolute constant $C$
\begin{equation}
\label{sumexp}
\int_0^1 |\sum_{\alpha\in A} c_{\alpha} e^{i\alpha y}|^2\, dy \le C\,S\,.
\end{equation}
\end{lemma}

Of course, one can change variables and get:
\begin{cor}
Let $j=1,2,...k$, $c_j\in\C$, $|c_j|=1$, and $\alpha_j\in\R$. Let $A:=\lbrace \alpha_j\rbrace_{j=1}^k$, and let $\delta>0$.
Suppose
\begin{equation}
\int_{\R} (\sum_{\alpha\in A} \chi_{[\alpha-\delta, \alpha+\delta]}(x))^2\,dx \le S\,,
\end{equation}
Then there exists an abolute constant $C$
\begin{equation}
\label{sumexp2}
\int_a^{a+\delta^{-1}} |\sum_{\alpha\in A} c_{\alpha} e^{i\alpha y}|^2\, dy \le C\,S\,/{\delta^2}.
\end{equation}
\end{cor}

\noindent{\bf Remark.} Lemma \ref{CETSQ} is obviously stronger than Lemma \ref{CET}. In fact, let $S_0$ be the maximal number of points $A$ in any unit interval. Then
$$f(x) :=\sum_{\alpha\in A} \chi_{[\alpha-1,\alpha+1]}(x)\le 2S_0.$$
Now $\int_{\R} f^2(x) dx\le 4kS_0$, where $k$ as above is the cardinality of $A$. We can put now $S:= 4kS_0$, apply Lemma \ref{CETSQ} and get the conclusion of Lemma \ref{CET}. The proof of Lemma \ref{CETSQ} does not require the Carleson imbedding theorem. Here it is.

\begin{proof}
Using Plancherel's theorem we write
$$
\int_0^1|\sum_{\alpha\in A} c_{\alpha} e^{i\alpha\,y}\, dy|^2 \le e\int_0^1|\sum_{\alpha\in A} c_{\alpha} e^{i(\alpha+i)\,y}\, dy|^2 \le e\int_0^{\infty}|\sum_{\alpha\in A} c_{\alpha} e^{i(\alpha+i)\,y}\, dy|^2=
$$
$$
e\int_{\R} \bigg|\sum_{\alpha\in A}\frac{c_{\alpha}}{\alpha +i -x}\bigg|^2\,dx\,.
$$

Recall that
\begin{equation}
\label{CETeq1}
 H^2(\C_+)\,\,\text{is orthogonal to}\,\, \overline{H^2(\C_+)}
\end{equation}

Now we continue
$$
\int_{\R} \bigg|\sum_{\alpha\in A}\frac{c_{\alpha}}{\alpha +i -x}\bigg|^2\,dx\le 
$$
$$
 \int_{\R} \bigg|\sum_{\alpha\in A}\frac{c_{\alpha}}{\alpha +i -x}-\sum_{\alpha\in A}\frac{c_{\alpha}}{\alpha -i -x}\bigg|^2\,dx=
 $$
 $$
 \frac{\pi}{2}\int_{\R} \bigg|\sum_{\alpha\in A}c_{\alpha}P_1(\alpha-x)\bigg|^2\,dx\,,
 $$
 where $P_1$ is the Poisson kernel in the half-plane $C_+$ at hight $h=1$:
 $$
 P_h(x):=\frac1{\pi}\frac{h}{h^2 +x^2}\,.
 $$
 We continue by noticing that $P_1*\chi_{[\lambda-1,\lambda+1]}(x) \ge c\,P_1(\lambda-x)$ with absolute positive $c$. This is an elementary calculation, or, if one wishes, Harnack's inequality. Now we can continue
$$
\int_0^1|\sum_{\alpha\in A} c_{\alpha} e^{i\alpha\,y}\, dy|^2 \le\frac{\pi e}{2c}\int_{\R} \bigg|(P_1*\sum_{\alpha\in A}c_{\alpha}\chi_{[\alpha-1,\alpha+1]})(x)\bigg|\,dx\,.
$$
Now we use the fact that $f\rightarrow P_1*f$ is a contraction in $L^2(\R)$. So

$$
 \int_0^1|\sum_{\alpha\in A} c_{\alpha} e^{i\alpha\,y}\, dy|^2 \le\frac{\pi e}{2c}\int_{\R} |\sum_{\alpha\in A}c_{\alpha}\chi_{[\alpha-1,\alpha+1]}(x)|^2\,dx\le C\, S\,.
 $$
 The lemma is proved.

\end{proof}

\subsection{{\bf A Blaschke estimate}}
\begin{lemma}\label{schke1}
Let $D$ be the closed unit disc in $\C$. Suppose $\phi$ is holomorphic in an open neighborhood of $D$, $|\phi(0)|\geq 1$, and the zeroes of $\phi$ in $\frac{1}{2}D$ are given by $\lambda_1,\lambda_2,...,\lambda_M$. Let $C=||\phi||_{L^\infty(D)}$. Then $M\leq \log_2(C).$
\end{lemma}
\begin{proof}
Let 
$$
B(z)=\prod_{k=1}^M{\frac{z-\lambda_k}{1-\bar{\lambda_k}z}}.
$$
 Then $|B|\leq 1$ on $D$, with $=$ on the boundary. If we let $g:=\frac{\phi}{B}$, then $g$ is holomorphic and nonzero on $\frac{1}{2}D,$ and $|g(e^{i\theta})|\leq C$ $\forall\theta\in [0,2\pi]$. Thus $|g(0)|\leq C$ by the maximum modulus principle. So we have $$C\geq |g(0)|=\frac{|\phi(0)|}{|B(0)|}\geq\prod_{k=1}^M{\frac{1}{|\lambda_k|}}\geq 2^M.$$
\end{proof}

\begin{lemma}\label{schke2}
In the same setting as Theorem $\ref{schke1}$, the following is also true for all $\delta\in (0,1/3)$: $\lbrace z\in\frac{1}{4}D:|\phi|<\delta\rbrace\subseteq\bigcup_{1\leq k\leq M} B(\lambda_k,\e)$, where $$\e:=\frac{9}{16}(3\delta)^{1/M}\leq\frac{9}{16}(3\delta)^{1/\log_2(C)}.$$
\end{lemma}
\begin{proof}
Let $\delta\in (0,1/3)$, and let $z\in\frac{1}{4}D$ such that $|z-\lambda_k|>\e\,\,\forall k$. Note that $g$ is harmonic and nonzero on $\frac{1}{2}D$ with $|g(0)|\geq 2^M$. Thus Harnack's inequality ensures that $|g|\geq\frac{1}{3}2^M$ on $\frac{1}{4}D$, so there 
$$
|\phi(z)|\geq |g(z)B(z)| \geq\frac{1}{3}2^M\prod_{k=1}^M{|\frac{z-\lambda_k}{1-\bar{\lambda_k}z}|}\geq(\frac{16\e}{9})^M\frac{1}{3}=\delta.
$$
 We can conclude the proof by the contrapositive.
\end{proof}

\section{Combinatorial theorem}
\label{combi}
For this section, regard the set $E$ from Section $\ref{sec:fourier}$ as parameterized by $\theta$, and use the variable $x$ instead of $s$ on the non-Fourier side, since we will not work on the Fourier side at all during this section.

\begin{theorem}
\label{combin}

 Let $\theta\in E$. Then
 $$
 \max_{n: 0\le n\le N} \|f_{n,\theta}\|^2_{L^2(\R)}\le C\, K\,.
 $$
\end{theorem}

To prove this we first need the following claim, which is the main combinatorial assertion of this article. It repeats the one in \cite{NPV} but we give a slightly different proof.

We fix a direction $\theta$, we think that the line $\ell_{theta}$ on which we project is $\R$. If $x\in \R$ then by $N_x$ we denote the line orthogonal to $\R$ and passing through point $x$, we call $N_x$  a needle. By $F_L$ we denote $\{x\in\R: f^*_N(x) :=\max_{0\le n\le N} f_{n,\theta}(x) >L\}$ (also known as $A_L^*$).

\begin{theorem}
\label{combinlemma}
There exists an absolute constant $C$ such that for any large $K$ and $M$
\begin{equation}
\label{ssim1}
|F_{4KM}|\le C\, K\, |F_K|\cdot |F_M|\,.
\end{equation}
\end{theorem}

\begin{proof}
This will be a proof by greedy algorithm. First choose $y\in F_{4K}$ and consider needle $N_y$ and triangles of certain size $3^{-j_y}, j_y\le N$ intersecting $N_y$. Consider any family of this sort having more than $4K$ elements. Fix such a family. We will ``fathorize" it, i.e. we consider the father of each element in the family. Two things may happen: 1) there are more than $4K$ distinct fathers; 2) number of fathers is at most $4K$. In the latter case the number of fathers is at least $2K$. In fact, we slash the number of elements by fathorizing, but not more than by factor of $1/2$. If the first case happens fathorize again, do this till we get to the second case. 

After doing this procedure with all $x\in F_{4K}$ and all families of cardinality bigger than $4K$ of equal size triangles  intersecting needle $N_x$ we come to some awfully complicated set of triangles. But we will consider now maximal-by-inclusion triangles of this family, the family of these maximal triangles is called $\mathcal{F}_0$.

\medskip

Choose triangle $Q_{00}\in \mathcal{F}_0$ such that its sidelength $\ell(Q_{00})$  is maximal possible in $\mathcal{F}_0$. It is very important to notice that $\mathcal{F}_0$ contains at least $2K-1$ triangles of the same size as $Q_{00}$ pierced by a needle $N_{y_0}$. This is because of maximality of the lengthsize, the stack pierced by $N_{y_0}$ could not be eaten up even partially by bigger in size triangles from some other stack. So let us call by $Q_{01},..., Q_{02K-1},..., Q_{0S}$, $S\ge 2K-1$. They are of the same size as $Q_{00}$ and all intersect a certain needle $N_{y_0}$. 

\medskip

Denote $$ I_0=\text{proj}\,Q_{00}\,.$$

\medskip

Consider all $q\in \mathcal{F}_0$ such that $$\text{proj}\,q\cap 20\,I_0\neq\emptyset\,.$$
Call them $\mathcal{F}(Q_{00})$. Of course $\ell(q) \le \ell(Q_{00})$. For every such $q$ consider
a Cantor square $Q$, $q\subset Q$, such that $\ell(Q)=\ell(Q_{00})$. Such $Q$'s form family $\tilde{\mathcal{F}}(Q_{00})$.

\begin{lemma}
\label{cl1}
For every $y\in \R$ the needle $N_y$ intersects at most $4K$ triangles of the family $\tilde{\mathcal{F}}(Q_{00})$.
\end{lemma}

\begin{proof}

Suppose contrary. Then $N_y$ intersects more than $4K$ of triangles from $\tilde{\mathcal{F}}(Q_{00})$. So $y\in F_{4K}$, and our pierced family is one of those which we considered at the begining. It can be fathorized. Then the square of size $\ge 2\, \ell(Q_{00})$ will be prrsent in $\mathcal{F}_0$. Contradiction with maximality of length.

\end{proof}

\begin{lemma}
\label{cl2}
$\text{card}\,\tilde{\mathcal{F}}(Q_{00})\le 88\,K\,.$
\end{lemma}

\begin{proof}
$$
\text{card}\,\tilde{\mathcal{F}}(Q_{00})\cdot \ell(Q_{00}) = \sum_{Q\in \tilde{\mathcal{F}}(Q_{00})}\ell(Q) \le 
$$
$$
\int_{22 I_0} \text{card}\, \{Q\in \tilde{\mathcal{F}}(Q_{00}): Q\cap N_y\neq \emptyset\}\,dy\le
$$
$$
4K\cdot 22\ell(Q_{00})\,.
$$
This is by Lemma \ref{cl1}.

\end{proof}

\begin{lemma}
\label{cl3}
There exists an interval $J_0\subset I_{y_0}$ such that $|J_0| \ge c\cdot |I_0|$ with a ceratin absolute positive $c$. And $J_0\subset F_{K}$.
\end{lemma}

\begin{proof}
We already noticed that $Q_{00}, Q_{01},..., Q_{02K-1}$ intersect needle $N_{y_0}$. Then at least half of them have their center of symmetry to the right of $N_{y_0}$, or at least half of them have their center of symmetry to the left of $N_{y_0}$. Assume that the first case occurs. Then the segment $[y_0, c\cdot \ell(Q_{00})]$ obviously is contained in $F_K$.

\end{proof}

\begin{lemma}
\label{cl4}
$|F_{4KM}\cap 20I_0| \le C\, K\, \ell(Q_{00})=C\,K\,|I_0|\,.$
\end{lemma}

\begin{proof}
Of course $F_{4KM}\subset F_{4K}$. For $y\in F_{4KM}\cap 20I_0$ the whole family of small triangles whose quantity is $> 4KM$ intersecting $N_y$ will be inside one of those $Q\in \tilde{\mathcal{F}}(Q_{00})$, whose number is at most $88K$ by Lemma \ref{cl2}. Let us enumerate $Q^1,..., Q^s$, $s\le 88K$ elements of $\tilde{\mathcal{F}}(Q_{00})$. So there exists $i=1,...,s$ such that
$$
y\in\text{dilated copy of}\,F_M\,\text{in}\,\text{proj}\,Q^i\,.
$$
Hence
$$
F_{4KM}\cap 20 I_0\subset \cup_{i=1}^{88K}\text{dilated copy of}\,F_M\,\text{in}\,\text{proj}\,Q^i\,.
$$
So
$$
|F_{4KM}\cap 20 I_0|\le \sum_{i=1}^{88K}\ell(Q^i)|F_M| \le 88K\, \ell(Q_{00})|F_M|\,.
$$

\end{proof}

\begin{lemma}
\label{cl5}
$|F_{4KM}\cap 20I_0| \le 88 c^{-1} K|F_m|\cdot |J_0|\,.$
\end{lemma}

Now we want to repeat all steps for $F_{4K}^0:= F_{4K} \setminus 20 I_0$. So we fathorize triangles peirced by needles $N_x$, $x\in F_{4K}^0$. As before we get families $\mathcal{F}_1$, maximal sidelength trinagle $Q_{11}$, families $\mathcal{F}(Q_{11})$, $\tilde{\mathcal{F}}(Q_{11})$. Notice that $\mathcal{F}_1<\mathcal{F}_0$ in the sense that for every $q\in \mathcal{F}_1$ there exists $q\in\mathcal{F}_0$ such that $q$ is contained in $Q$. It is also clear that 
$$
\ell(Q_{11})\le \ell(Q_{00})\,.
$$
Obviously $Q_{00}, Q_{01},...$ are not in $\mathcal{F}_1$, their projections even do not intersect $\R\setminus 20 I_0$.

There are at least $2K-1$ brothers of $Q_{11}$: $Q_{12},...,Q_{1 2K-1},...$ in $\mathcal{F}_1$ such that
they are of the same size $\ell(Q_{11})$ and they (and $Q_{11}$) intersect the same needle $N_{y_1}$, $y_1\in \R\setminus 20I_0$.
This is again the maximality of the sidelength among $\mathcal{F}_1$ triangles. Let $I_1:= \text{proj}\, Q_{11}$.
Notice that
$$
I_1 \cap I_0=\emptyset\,.
$$
In fact, $y_1 \in I_1, y_1\notin 20I_0$, $Q_{11}$ size is much smaller than $20|I_0|$. We consider all $q\in \mathcal{F}_1$ such that
$$
\text{proj}\, q\cap (20I_1\setminus 20I_0) \neq \emptyset\,.
$$ 
Call this family $\mathcal{F}(Q_{11})$. For every $q\in \mathcal{F}(Q_{11})$ consider Cantor triangle $Q$ containing $q$ and of the size $\ell_1=\ell(Q_{11})$. Maximal-by-inclusion among such $Q$'s form $\tilde{\mathcal{F}}(Q_{11})$.

\begin{lemma}
\label{cl1prime}
For any $y\in R\setminus 20I_0$, $N_y$ intersects at most $4K$ triangles of $\tilde{\mathcal{F}}(Q_{11})$.
\end{lemma}

\begin{proof}
Suppose contrary. Then there exists $y_1'\in F_{4K}\cap (\R\setminus 20I_0)$, and a subfamily of  $\tilde{\mathcal{F}}(Q_{11})$ of cardinality bigger than $4K$ intersects $N_{y_1'}$. It can be fathorized. Then triangles of size $\ge 2\ell(Q_{11})$ would belong to $\mathcal{F}_1$. This contradicts the maximality of $\ell(Q_{11})$.

\end{proof}

\begin{lemma}
\label{cl11}
For any $z\in \R$, $N_z$ intersects at most $8K$ triangles of $\tilde{\mathcal{F}}(Q_{11})$.
\end{lemma}

\begin{proof}
Suppose contrary. Then there exists $z\in F_{4K}$, and a subfamily of  $\tilde{\mathcal{F}}(Q_{11})$ of cardinality bigger than $4K$ intersects $N_{z}$. Now there is an end-point of $20I_1\setminus 20I_0$ (call it $a$), which is closest to $z$. Let it be on the right of $z$. Then another end-point is also on the right but farther away. As every traingle from the family has a) $z$ in its projection, and b) a ceratin point to the right of $a$ in its projection (their projections intersect $20I_1\setminus 20I_0$--by definition), then all of them have $a$ in its projection. Let us be lavish and say that $50$ percent of them have $a$ in their projection (the fact is that it is not lavishness, it is necessity: next step will be to consider in the future $20I_2\setminus (20I_0\cup 20I_1)$, and their can be $2$ closest points to $z$: one on the left, say, $b$, and one on the right, say, $a$, and we can guarantee that $50$ percent of our triangles have either $b$ or $a$ in their projections simultaneously). We use the previous Lemma \ref{cl1prime}, and get that this $5)$ percent is $\le 4K$. So we are done.

\end{proof}

\begin{lemma}
\label{cl2prime}
$\text{card} \, \tilde{\mathcal{F}}(Q_{11}) \le 172 K\,.$
\end{lemma}

\begin{proof}
$$
\text{card}\,\tilde{\mathcal{F}}(Q_{11})\cdot \ell(Q_{11}) = \sum_{Q\in \tilde{\mathcal{F}}(Q_{11})}\ell(Q) \le 
$$
$$
\int_{22 I_1} \text{card}\, \{Q\in \tilde{\mathcal{F}}(Q_{11}): Q\cap N_y\neq \emptyset\}\,dy\le
$$
$$
8K\cdot 22\ell(Q_{11})\,.
$$
This is by Lemma \ref{cl1}.

\end{proof}

\begin{lemma}
\label{cl3prime}
There exists an interval $J_1\subset I_1$, $|J_1|\le c\cdot |I_1|$, such that $J_1\subset F_K$.
\end{lemma}

\begin{proof}
The same proof as for Lemma \ref{cl3}.
\end{proof}

\begin{lemma}
\label{cl4prime}
$|F_{4KM}^0\cap 20 I_1| \le C\, K\, \ell(Q_{11}\le C|, K\, |I_1|\,.$
\end{lemma}

\begin{proof}

The same proof as for Lemma \ref{cl4}.
\end{proof}

Combining Lemmas \ref{cl3prime}, \ref{cl4prime} we get

\begin{lemma}
\label{cl5prime}
$|F_{4KM}^0\cap 20 I_1| \le  C\, c^{-1}\, K\, |J_1|\,.$
\end{lemma}

We continue by introducing 
$$
F^1_{4KM}= F_{4KM} \setminus (20 I_0\cup 20 I_1)\,.
$$
We repeat the whole procedure. There will be $I_2$, $J_2\subset I_2 \cap F_K, |J_2|\ge c\cdot |I_2|$:
$$
I_2 \cap (I_1\cup I_0) =\emptyset\,,
$$
$$
|F_{4KM} \cap 20 I_2| \le Cc^{-1}K|J_2||F_M|\,,
$$
et cetera. 

Finally,
$$
|F_{4KM}|\le |F_{4KM}\cap 20 I_0|+ |(F_{4KM}\setminus 20 I_0)\cap 20 I_1|+...+|(F_{4KM}\setminus 20 I_0\cup 20I_1\cup....20I_{j-1})\cap 20 I_j|+...\le
$$
$$
C'\,K\,|F_M|\sum_{j=0}^{\infty}|J_j| \le C'\,K\,|F_M|\,|F_K|\,.
$$
We are done with Theorem \ref{combinlemma}.
\end{proof}

\medskip

Now we can prove Theorem \ref{combin}.

\begin{proof}
Let $E_j:=\{x: f_{n,\theta}(x) >(4K)^{j+1}\}$, $j=0,1,... .$.
We know by Theorem \ref{combinlemma} that
$$
|E_j| \le (CK)^j|E_0|^{j+1}\,.
$$
Hence,
$$
\int f_{n, \theta}(x)^2\,dx \le 4K \int f_{n,\theta}(x)\,dx + \sum_{j+0}^{\infty}\int_{E_j\setminus E_{j+1}}f_{n, \theta}(x)^2\,dx \le
$$
$$
4CK + (4K)^{j+2} \,(CK)^j|E_0|^{j+1}\,.
$$
If $|\{x: f^*_N(x) >K\}| \le 1/K^{2+\tau}$ then for all $n\le N$ we can immediately read the previous inequality as 
$$
\int f_{n, \theta}(x)^2\,dx  \le C(\tau) \,K\,.
$$

\end{proof}

\section{Discussion}
\label{discu}

\subsection{{\bf Difficulties for more general self-similar sets}}\label{general sets}
 Analytic tiling in every direction is unique to the gasket, though perhaps there is some hope that something similar occurs for typical directions in the arbitrary case. Suppose we had 5 self-similarities, and that for for some direction $\theta$, we had $\phi_\theta (x_0)= 1 + (-i) + i + e^{2\pi i/3} + e^{4\pi i/3}=0$. Then clearly, taking fifth powers of the summands results in another zero with exactly the same summands, in complete and utter contrast to the three-point case. Similar examples using partitions into relatively prime roots of unity exist for numbers other than 5.
 
 At any rate, our arguments without analytic tiling can still get the estimate $Fav(\G_n)\leq e^{-c\sqrt{\log\,n}}$. It appears that the above approach will work for some more general self-similar sets, but new ideas are needed if one is to get better upper bounds than $e^{-c\sqrt{\log\,n}}$.
 
 Even to get this weak upper bound for more general sets, one has to deal with branching points, which certainly can exist, but even then the order of the zeroes of $\phi_\theta$ will be controlled by the number of terms in $\phi_\theta$, i.e., the number of similarity maps. Some more advanced lemmas like those of Turan or Tijdeman can help control the size of the set where $\phi_\theta$ is small.
 
\subsection{{\bf An estimate for degenerate gaskets}} \label{degenerate gaskets}
\newcommand{\xtil}{\tilde{x}}

Fix two self-similarity centers $p_1,p_3$ with $|p_1-p_3|=1$ and choose a third self-similarity center $p_2$ so that $|p_2-p_3|,|p_2-p_1|\leq 1$. Define the \textbf{degeneracy} $\delta$ of this configuration (and the resulting gasket) to be twice the area of the triangle with corners $p_j$. In particular, if one fixes $p_1=0,p_3=1,p_3=1/2+2i\delta$, then $G$ approaches a line segment as $\delta$ approaches $0$, and in general, any upper bound on $Fav(\G_n)$ should break down as $\delta\to 0$. We will show this by highligthing the places where $\delta$ makes a difference.

For $j=1,2,3$, write $p_j=re^{i\theta_j}$. This can be done with $r$ not depending on $j$, since three non-colinear points define a circle. For $\theta\in [0,2\pi]$, let $c_j:=r\cos(\theta_j-\theta)$, and let $s_j:=r\sin(\theta_j-\theta)$. One gets
$$\phi_\theta(x)=\frac1{3}\sum_{j=1}^3 e^{-ic_jx}$$
$$\varphi_t(\xtil)=e^{-i\frac{c_1}{c_3-c_1}\xtil}(1+e^{-it\xtil}+e^{-i\xtil}),$$
where $\xtil=(c_3-c_1)x$ and $t=\frac{c_2-c_1}{c_3-c_1}$. We can consider this with the indices permuted, so that we are always in the case $c_1\leq c_2\leq c_3$. Thus

$$\delta ||\phi_\theta||_{L^2(x)} \leq ||\varphi_t||_{L^2(\xtil)}\leq ||\phi_\theta||_{L^2(x)}.$$

Lemma $\ref{CETSQ}$ gains a $\delta^{-1}$ on the right-hand side wherever it is applied, since the frequencies might be packed in a lot tighter. Thus Proposition $\ref{P1below}$ is the same $$\int_I|P_{1,t}(x)|^2dx\geq C3^m,$$
but it is not true unless $CK/\delta\leq 3^{3m/4}$ for all $N\geq N^*$. So $N^*\geq C\delta^{-1}$.

Propositions $\ref{Pshpest}$ also picks up a $\delta^{-1}$ on the right hand side, since $P_1^{\sharp}$ used Lemma $\ref{CETSQ}$ as well.

The final estimate becomes

$$\int_{\tilde{E}}\int_{SSV(t)}|P_{1,t}(x)|^2dx\,dt\leq 3^{Bm}CmK3^m\cdot C''m3^m(\frac{7}{9})^{\ell}\cdot\delta^{-1},$$

so the $(\frac{7}{9})^\ell$ term has to work that much harder. But easily the right hand side is at most $3^{m/4}\delta^{-1}\leq 3^{m/2},$

which again is no trouble if $N^*\geq\delta^{-2}$.

Now we need to deal with the change of variables $\theta\to t$.

\begin{lemma}
$$\delta\leq |\frac{dt}{d\theta}| \leq\delta^{-1}$$
\begin{proof}
Remember, we are working in a relabeling where $c_1\leq c_2\leq c_3$
$$\frac{dt}{d\theta}=\frac{(c_3-c_1)(s_2-s_1)-(c_2-c_1)(s_3-s_1)}{(c_3-c_1)^2}.$$
Note that the numerator is constant\footnote{Differentiate without multiplying anything out}. So if you evaluate this constant when $c_1=c_2$, the surviving term is exactly $\delta$, the base times the height of the triangle. The result follows, since $\delta^2\leq (c_3-c_1)^2\leq 1$
\end{proof}
\end{lemma}

Let

$$f_{N,\theta}^*(s):=\sup_{n\le N} f_{n,\theta}(s)$$
$$A^*_{K,\theta}:=\{s:f_{N,\theta}^*(s)\geq K\}$$
$$
E_{\theta}:=\lbrace \theta : |A^*_{K,\theta}|\leq \frac1{K^3}\rbrace\,.
$$

Thus $|E_{\theta}|\leq |E|\delta^{-1}\leq\frac1{\delta K}$.

Putting everything together with Section $\ref{combinatorics}$, we get

$$Fav(\G_{NK^3})\leq\text{average length for good angles} + 3|\text{bad angles}|\leq\frac{C}{K} + \frac{3}{\delta K}$$
$N=K^{1/\epsilon_0}$, so
$$Fav(\G_{K^{3+1/\epsilon_0}})\leq \frac{C}{\delta K},\text{ or}$$
$$Fav(\G_M)\leq\frac{C}{\delta M^{\epsilon_0/(3\epsilon_0+1)}},$$

so long as $M\geq\delta^{-2(1+3\epsilon_0)}$. Otherwise, we have the upper bound $3$, and the bound is valid for all cases. So for all $n$, we can write

$$Fav(\G_n)\leq \frac{C_{\epsilon_0}}{\delta n^{\epsilon_0/(1+3\epsilon_0)}}$$

By using $\beta=2+\eta$ instead of $\beta=3$, one can get

\begin{equation}\label{p and delta}Fav(\G_n)\leq \frac{C_{\epsilon_0}}{\delta n^{\epsilon_0/(1+(2+\eta)\epsilon_0)}}.\end{equation}

\subsection{{\bf The heart of the dragon}}\label{dragon}

There is a fable about dragonslaying. To slay the dragon, one must destroy the heart. The heart is inside of a tetrahedron, which is inside of a cube, which is inside of an octahedron, which is inside of a dodecahedron, which lies inside of an icosahedron. There are only 5 Platonic solids, so there is a limit to how convoluted such a story can get, but the story is sufficiently convoluted. So it is with value of $p$ in the main theorem, the dragon exponent. Let us trace the dependences here.

$p$ depends of course on equation $\ref{p and delta}$, which depends on $\epsilon_0$. In turn, $\epsilon_0$ is determined by equation $\ref{totalest}$. Here, an improvement is possible; one only needs $m=(1+\eta)\epsilon_0\,\log\,N$. So one can take $\epsilon_0<\frac1{2A}$.

Recall that $SSV(t)$ is where $|P_2|<3^{-Am}=(\varepsilon^*/9)^m$. (See Proposition $\ref{P1onSSV}$). So we saw in Corollary $\ref{SSV final}$ that $\varepsilon^*<3^{-\alpha M}$ is sufficient. So $A>\alpha M+2$.

So now our quest for the dragon's heart meets a fork in the road. To get $M$, one must go through Section $\ref{complex}$ with more care. $M$ depends on $H$; one can take the largest integer $M$ such that $M\leq\log_2(max_{|y|<H,|x|<3^{m},|im(t)|<3^{-m}}|\varphi_t(x+iy)|)+1$, and $max|\varphi|<(e+1)e^H$. $H=2.4$ is sufficient for all considerations, so $M\leq 5$.

Next, there is $\alpha$. Note that $Cm3^m(7/9)^\ell K3^m\cdot 3^{3m}\leq K$ at the end of Section $\ref{sec:fourier}$ (the $3^{3m}$ is gotten from summing over $j$ and $r$). Since $\ell=\alpha m$, we need $\alpha>\frac{5}{\log_3(9/7)}<21.86$.

Therefore, $A>21.86\cdot 5+2<111.29$, and $\epsilon_0<\frac1{2A}$, or sufficiently, $\epsilon_0<\frac1{223}$ and $p<\frac1{225}$.

  \bibliographystyle{amsplain}

\end{document}